\documentclass{amsart}
\usepackage{amsmath}
\usepackage{amssymb}
\usepackage{amsfonts}

\newtheorem{theorem}{Theorem}
\theoremstyle{plain}

\newtheorem{corollary}{Corollary}

\newtheorem{lemma}{Lemma}

\newtheorem{proposition}{Proposition}
\newtheorem{remark}{Remark}

\numberwithin{equation}{section}
\input{tcilatex}

\begin{document}
\title[On the integral inequalities]{On the Integral inequalities for
mappings whose second derivatives are convex and applications }
\author{Mehmet Zeki Sar\i kaya$^{\star \clubsuit }$}
\address{$^{\clubsuit }$Department of Mathematics,Faculty of Science and
Arts, D\"{u}zce University, D\"{u}zce, Turkey}
\email{sarikayamz@gmail.com}
\thanks{$^{\star }$corresponding author}
\author{Erhan. SET$^{\blacksquare }$}
\address{$^{\blacksquare }$Atat\"{u}rk University, K.K. Education Faculty,
Department of Mathematics, 25240, Campus, Erzurum, Turkey}
\email{erhanset@yahoo.com}
\author{M. Emin Ozdemir$^{\blacklozenge }$}
\address{$^{\blacklozenge }$Graduate School of Natural and Applied Sciences,
A\u{g}r\i\ \.{I}brahim \c{C}e\c{c}en University, A\u{g}r\i , Turkey}
\email{emos@atauni.edu.tr}
\date{}

\begin{abstract}
In this paper, we establish several new inequalities for some twice
differantiable mappings. Then, we apply these inequalities \ to obtain \ new
midpoint, trapezoid and perturbed trapezoid rules. Finally, some
applications for special means of real numbers are provided.
\end{abstract}

\maketitle

\section{Introduction}

In 1938 Ostrowski obtained a bound for the absolute value of the difference
of a function to its average over a finite interval. The theorem is well
known in the literature as Ostrowski's integral inequality \cite{Ostrowski}:

\begin{theorem}
Let $f:[a,b]\mathbb{\rightarrow R}$ be a differentiable mapping on $(a,b)$
whose derivative $f^{\prime }:(a,b)\mathbb{\rightarrow R}$ is bounded on $%
(a,b),$ i.e., $\left\Vert f^{\prime }\right\Vert _{\infty }=\underset{t\in
(a,b)}{\overset{}{\sup }}\left\vert f^{\prime }(t)\right\vert <\infty .$
Then, the inequality holds:%
\begin{equation}
\left\vert f(x)-\frac{1}{b-a}\int\limits_{a}^{b}f(t)dt\right\vert \leq \left[
\frac{1}{4}+\frac{(x-\frac{a+b}{2})^{2}}{(b-a)^{2}}\right] (b-a)\left\Vert
f^{\prime }\right\Vert _{\infty }  \label{1}
\end{equation}%
for all$\ x\in \lbrack a,b].$ The constant $\frac{1}{4}$ is the best
possible.
\end{theorem}

In 1976, Milovanovic and Pecaric proved a generalization of the Ostrowski
inequality for $n$-times differentiable mappings (see for example \cite[p.468%
]{Milovanovic}). Dragomir and Wang (\cite{DW}, \cite{DW1}) extended the
result (\ref{1}) and applied the extended result to numerical quadrature
rules and to the estimation of error bounds for some special means. Also,
Sofo and Dragomir \cite{Sofo} extended the result (\ref{1}) in the $L_{p}$
norm. Dragomir (\cite{D1}-\cite{D3})\ further extended the (\ref{1}) to
incorporate mappings of bounded variation, Lipschitzian and monotonic
mappings. For recent results and generalizations conserning Ostrowski's
integral inequality see \cite{AD}-\cite{Milovanovic}, \cite{Sofo}, \cite%
{sarikaya}, and the references therein.

In \cite{Cerone1}, Cerone and Dragomir find the following perturbed
trapezoid inequalities:

\begin{theorem}
Let $f:[a,b]\mathbb{\rightarrow R}$ be such that the derivative $f^{\prime }$
is absolutely continuous on $[a,b].$ Then, the inequality holds:%
\begin{eqnarray}
&&\left\vert \int\limits_{a}^{b}f(t)dt-\frac{b-a}{2}\left[ f(b)+f(a)\right] +%
\frac{\left( b-a\right) ^{2}}{8}\left[ f^{\prime }(b)-f^{\prime }(a)\right]
\right\vert  \notag \\
&&  \label{01} \\
&\leq &\left\{ 
\begin{array}{l}
\frac{(b-a)^{3}}{24}\left\Vert f^{\prime \prime }\right\Vert _{\infty }\text{
\ \ \ if \ \ }f^{\prime \prime }\in L_{\infty }[a,b] \\ 
\\ 
\frac{(b-a)^{2+\frac{1}{q}}}{8(2q+1)^{\frac{1}{q}}}\left\Vert f^{\prime
\prime }\right\Vert _{p}\text{ \ \ \ if \ \ }f^{\prime \prime }\in
L_{p}[a,b],\ p>1,\ \frac{1}{p}+\frac{1}{q}=1 \\ 
\\ 
\frac{(b-a)^{2}}{8}\left\Vert f^{\prime \prime }\right\Vert _{1}\text{ \ \ \
if \ \ }f^{\prime \prime }\in L_{1}[a,b]%
\end{array}%
\right.  \notag
\end{eqnarray}%
for all$\ t\in \lbrack a,b].$
\end{theorem}

In recent years a number of authors have considered an error analysis for
some known and some new quadrature formulas. They used an approach from the
inequalities point of view. For example, the midpoint quadrature rule is
considered in \cite{Cerone1},\cite{USK},\cite{PP}, the trapezoid rule is
considered in \cite{Cerone1},\cite{USKMEO},\cite{SSO}. In most cases
estimations of errors for these quadrature rules are obtained by means of
derivatives and integrands.

In this article, we first derive a general integral identity for twice
derivatives functions. Then, we apply this identity to obtain our results
and using functions whose twice derivatives in absolute value at certain
powers are convex, we obtained new inequalities related to the Ostrowski's \
type inequality. Finally, we gave some applications for special means of
real numbers.

\section{Main Results}

In order to prove our main results, we need the following Lemma (see, \cite%
{Liu}):

\begin{lemma}
\label{lm} Let $f:I\subset \mathbb{R}\rightarrow \mathbb{R}$ be a twice
differentiable function on $I^{\circ }$ with $f^{\prime \prime }\in
L_{1}[a,b]$, then%
\begin{eqnarray}
&&\dfrac{1}{b-a}\dint_{a}^{b}f(u)du-\frac{1}{2}\left[ f(x)+f(a+b-x)\right] +%
\frac{1}{2}(x-\frac{a+3b}{4})\left[ f^{\prime }(x)-f^{\prime }(a+b-x)\right]
\notag \\
&&  \label{2} \\
&=&\frac{\left( b-a\right) ^{2}}{2}\dint_{0}^{1}k\left( t\right) f^{\prime
\prime }(ta+(1-t)b)dt  \notag
\end{eqnarray}%
where%
\begin{equation*}
k(t):=\left\{ 
\begin{array}{ll}
t^{2}, & 0\leq t<\frac{b-x}{b-a} \\ 
&  \\ 
\left( t-\frac{1}{2}\right) ^{2}, & \frac{b-x}{b-a}\leq t<\frac{x-a}{b-a} \\ 
&  \\ 
\left( t-1\right) ^{2}, & \frac{x-a}{b-a}\leq t\leq 1%
\end{array}%
\right.
\end{equation*}%
for any $x\in \left[ \frac{a+b}{2},b\right] .$
\end{lemma}

\begin{proof}
It suffices to note that%
\begin{eqnarray*}
I &=&\dint_{0}^{1}k\left( t\right) f^{\prime \prime }(ta+(1-t)b)dt \\
&& \\
&=&\dint_{0}^{\frac{b-x}{b-a}}t^{2}f^{\prime \prime }(ta+(1-t)b)dt+\dint_{%
\frac{b-x}{b-a}}^{\frac{x-a}{b-a}}\left( t-\frac{1}{2}\right) ^{2}f^{\prime
\prime }(ta+(1-t)b)dt \\
&& \\
&&+\dint_{\frac{x-a}{b-a}}^{1}\left( t-1\right) ^{2}f^{\prime \prime
}(ta+(1-t)b)dt \\
&& \\
&=&I_{1}+I_{2}+I_{3}.
\end{eqnarray*}%
By inegration by parts, we have the following identity%
\begin{eqnarray*}
I_{1} &=&\dint_{0}^{\frac{b-x}{b-a}}t^{2}f^{\prime \prime }(ta+(1-t)b)dt \\
&& \\
&=&\dfrac{t^{2}}{\left( a-b\right) }f^{\prime }(ta+(1-t)b)\underset{0}{%
\overset{\frac{b-x}{b-a}}{\mid }}-\dfrac{2}{a-b}\dint_{0}^{\frac{b-x}{b-a}%
}tf^{\prime }(ta+(1-t)b)dt \\
&& \\
&=&\dfrac{1}{\left( a-b\right) }\left( \frac{b-x}{b-a}\right) ^{2}f^{\prime
}(x)-\dfrac{2}{a-b}\left[ \dfrac{t}{\left( a-b\right) }f(ta+(1-t)b)\underset{%
0}{\overset{\frac{b-x}{b-a}}{\mid }}-\dfrac{1}{a-b}\dint_{0}^{\frac{b-x}{b-a}%
}f(ta+(1-t)b)dt\right] \\
&& \\
&=&-\dfrac{\left( b-x\right) ^{2}}{\left( b-a\right) ^{3}}f^{\prime }(x)-%
\dfrac{2\left( b-x\right) }{\left( b-a\right) ^{3}}f(x)+\dfrac{2}{\left(
b-a\right) ^{2}}\dint_{0}^{\frac{b-x}{b-a}}f(ta+(1-t)b)dt.
\end{eqnarray*}%
Similarly, we observe that%
\begin{eqnarray*}
I_{2} &=&\dint_{\frac{b-x}{b-a}}^{\frac{x-a}{b-a}}\left( t-\frac{1}{2}%
\right) ^{2}f^{\prime \prime }(ta+(1-t)b)dt \\
&& \\
&=&\dfrac{\left( a+b-2x\right) ^{2}}{4\left( b-a\right) ^{3}}\left[
f^{\prime }(x)-f^{\prime }(a+b-x)\right] +\dfrac{\left( a+b-2x\right) }{%
\left( b-a\right) ^{3}}\left[ f(x)+f(a+b-x)\right] \\
&& \\
&&+\dfrac{2}{\left( b-a\right) ^{2}}\dint_{\frac{b-x}{b-a}}^{\frac{x-a}{b-a}%
}f(ta+(1-t)b)dt.
\end{eqnarray*}%
and%
\begin{eqnarray*}
I_{3} &=&\dint_{\frac{x-a}{b-a}}^{1}\left( t-1\right) ^{2}f^{\prime \prime
}(ta+(1-t)b)dt \\
&& \\
&=&\dfrac{\left( b-x\right) ^{2}}{\left( b-a\right) ^{3}}f^{\prime }(a+b-x)-%
\dfrac{2\left( b-x\right) }{\left( b-a\right) ^{3}}f(a+b-x)+\dfrac{2}{\left(
b-a\right) ^{2}}\dint_{\frac{x-a}{b-a}}^{1}f(ta+(1-t)b)dt.
\end{eqnarray*}%
Thus, we can write%
\begin{eqnarray*}
I &=&I_{1}+I_{2}+I_{3} \\
&& \\
&=&\frac{1}{(b-a)^{2}}(x-\frac{a+3b}{4})\left[ f^{\prime }(x)-f^{\prime
}(a+b-x)\right] -\frac{1}{(b-a)^{2}}\left[ f(x)+f(a+b-x)\right] \\
&& \\
&&+\frac{2}{(b-a)^{2}}\dint_{0}^{1}f(ta+(1-t)b)dt.
\end{eqnarray*}%
Using the change of the variable $u=ta+(1-t)b$ for $t\in \left[ 0,1\right] $
and by multiplying the both sides by $\left( b-a\right) ^{2}/2$ which gives
the required identity (\ref{2}).
\end{proof}

Now, by using the above lemma, we prove our main theorems:

\begin{theorem}
\label{thm1} Let $f:I\subset \mathbb{R}\rightarrow \mathbb{R}$ be a twice
differentiable function on $I^{\circ }$ such that $f^{\prime \prime }\in
L_{1}[a,b]$ where $a,b\in I,$ $a<b$. If $\left\vert f^{\prime \prime
}\right\vert $ is convex on $[a,b],$ then the following inequality holds:%
\begin{eqnarray}
&&\left\vert \dfrac{1}{b-a}\dint_{a}^{b}f(u)du-\frac{1}{2}\left[
f(x)+f(a+b-x)\right] +\frac{1}{2}(x-\frac{a+3b}{4})\left[ f^{\prime
}(x)-f^{\prime }(a+b-x)\right] \right\vert  \notag \\
&&  \label{4} \\
&\leq &\dfrac{1}{\left( b-a\right) }\left[ \left( b-x\right) ^{3}+\left( x-%
\frac{a+b}{2}\right) ^{3}\right] \left( \frac{\left\vert f^{\prime \prime
}(a)\right\vert +\left\vert f^{\prime \prime }(b)\right\vert }{6}\right) 
\notag
\end{eqnarray}%
for any $x\in \left[ \frac{a+b}{2},b\right] .$
\end{theorem}

\begin{proof}
From Lemma $\ref{lm}$ and by the definition $k(t),$ we get%
\begin{eqnarray}
&&\left\vert \dfrac{1}{b-a}\dint_{a}^{b}f(u)du-\frac{1}{2}\left[
f(x)+f(a+b-x)\right] +\frac{1}{2}(x-\frac{a+3b}{4})\left[ f^{\prime
}(x)-f^{\prime }(a+b-x)\right] \right\vert  \notag \\
&&  \notag \\
&\leq &\frac{\left( b-a\right) ^{2}}{2}\dint_{0}^{1}\left\vert k\left(
t\right) \right\vert \left\vert f^{\prime \prime }(ta+(1-t)b)\right\vert dt 
\notag \\
&&  \notag \\
&=&\frac{\left( b-a\right) ^{2}}{2}\left\{ \dint_{0}^{\frac{b-x}{b-a}%
}t^{2}\left\vert f^{\prime \prime }(ta+(1-t)b)\right\vert dt+\dint_{\frac{b-x%
}{b-a}}^{\frac{x-a}{b-a}}\left( t-\frac{1}{2}\right) ^{2}\left\vert
f^{\prime \prime }(ta+(1-t)b)\right\vert dt\right.  \notag \\
&&  \label{5} \\
&&\left. +\dint_{\frac{x-a}{b-a}}^{1}\left( t-1\right) ^{2}\left\vert
f^{\prime \prime }(ta+(1-t)b)\right\vert dt\right\}  \notag \\
&&  \notag \\
&=&\frac{\left( b-a\right) ^{2}}{2}\left\{ J_{1}+J_{2}+J_{3}\right\} . 
\notag
\end{eqnarray}%
Investigating the three separate integrals, we may evaluate as follows:

By the convexity of $\left\vert f^{\prime \prime }\right\vert $, we arrive at%
\begin{eqnarray*}
J_{1} &\leq &\dint_{0}^{\frac{b-x}{b-a}}\left( t^{3}\left\vert f^{\prime
\prime }(a)\right\vert +(t^{2}-t^{3})\left\vert f^{\prime \prime
}(b)\right\vert \right) dt \\
&& \\
&=&\frac{\left( b-x\right) ^{4}}{4(b-a)^{4}}\left\vert f^{\prime \prime
}(a)\right\vert +\left( \frac{\left( b-x\right) ^{3}}{3(b-a)^{3}}-\frac{%
\left( b-x\right) ^{4}}{4(b-a)^{4}}\right) \left\vert f^{\prime \prime
}(b)\right\vert ,
\end{eqnarray*}%
\begin{eqnarray*}
J_{2} &\leq &\dint_{\frac{b-x}{b-a}}^{\frac{x-a}{b-a}}\left[ \left( t-\frac{1%
}{2}\right) ^{2}t\left\vert f^{\prime \prime }(a)\right\vert +\left( t-\frac{%
1}{2}\right) ^{2}(1-t)\left\vert f^{\prime \prime }(b)\right\vert \right] dt
\\
&& \\
&=&\frac{1}{3(b-a)^{3}}\left( x-\frac{a+b}{2}\right) ^{3}\left\vert
f^{\prime \prime }(a)\right\vert +\frac{1}{3(b-a)^{3}}\left( x-\frac{a+b}{2}%
\right) ^{3}\left\vert f^{\prime \prime }(b)\right\vert ,
\end{eqnarray*}%
\begin{eqnarray*}
J_{3} &\leq &\dint_{\frac{x-a}{b-a}}^{1}\left[ \left( t-1\right)
^{2}t\left\vert f^{\prime \prime }(a)\right\vert +(1-t)^{3}\left\vert
f^{\prime \prime }(b)\right\vert \right] dt \\
&& \\
&=&\left( \frac{\left( b-x\right) ^{3}}{3(b-a)^{3}}-\frac{\left( b-x\right)
^{4}}{4(b-a)^{4}}\right) \left\vert f^{\prime \prime }(a)\right\vert +\frac{%
\left( b-x\right) ^{4}}{4(b-a)^{4}}\left\vert f^{\prime \prime
}(b)\right\vert .
\end{eqnarray*}%
By rewrite $J_{1},J_{2},J_{3}$ in (\ref{5}), we obtain (\ref{4}) which
completes the proof.
\end{proof}

\begin{corollary}[Perturbed Trapezoid inequality]
\label{c1} Under the assumptions Theorem $\ref{thm1}$ with $x=b,$ we have%
\begin{equation*}
\left\vert \int\limits_{a}^{b}f(u)du-\frac{b-a}{2}\left[ f(b)+f(a)\right] +%
\frac{\left( b-a\right) ^{2}}{8}\left[ f^{\prime }(b)-f^{\prime }(a)\right]
\right\vert \leq \dfrac{(b-a)^{3}}{48}(\left\vert f^{\prime \prime }\left(
a\right) \right\vert +\left\vert f^{\prime \prime }\left( b\right)
\right\vert ).
\end{equation*}
\end{corollary}

\begin{remark}
We choose $\left\vert f^{\prime \prime }(x)\right\vert \leq M,\ \ M>0$ in
Corollary \ref{c1}, then we recapture the first part of the inequality (\ref%
{01}).
\end{remark}

\begin{corollary}[Trapezoid inequality]
Under the assumptions Theorem $\ref{thm1}$ with $x=b$ and $f^{\prime
}(a)=f^{\prime }(b)$ in Theorem $\ref{thm1},$ we have 
\begin{equation}
\left\vert \dfrac{1}{b-a}\dint_{a}^{b}f(u)du-\dfrac{f\left( a\right)
+f\left( b\right) }{2}\right\vert \leq \dfrac{(b-a)^{2}}{48}(\left\vert
f^{\prime \prime }\left( a\right) \right\vert +\left\vert f^{\prime \prime
}\left( b\right) \right\vert ).  \label{H9}
\end{equation}
\end{corollary}

\begin{corollary}[Midpoint inequality]
Under the assumptions Theorem $\ref{thm1}$ with $x=\frac{a+b}{2}$ in Theorem 
$\ref{thm1},$ we have 
\begin{equation}
\left\vert \dfrac{1}{b-a}\dint_{a}^{b}f(u)du-f\left( \frac{a+b}{2}\right)
\right\vert \leq \dfrac{(b-a)^{2}}{48}(\left\vert f^{\prime \prime }\left(
a\right) \right\vert +\left\vert f^{\prime \prime }\left( b\right)
\right\vert ).  \label{H10}
\end{equation}
\end{corollary}

Another similar result may be extended in the following theorem

\begin{theorem}
\label{thm2} Let $f:I\subset \mathbb{R}\rightarrow \mathbb{R}$ be a twice
differentiable function on $I^{\circ }$ such that $f^{\prime \prime }\in
L_{1}[a,b]$ where $a,b\in I,$ $a<b$. If $\left\vert f^{\prime \prime
}\right\vert ^{q}$ is convex on $[a,b],$\ $q>1$ and $\frac{1}{p}+\frac{1}{q}%
=1,$ then%
\begin{eqnarray}
&&\left\vert \dfrac{1}{b-a}\dint_{a}^{b}f(u)du-\frac{1}{2}\left[
f(x)+f(a+b-x)\right] +\frac{1}{2}(x-\frac{a+3b}{4})\left[ f^{\prime
}(x)-f^{\prime }(a+b-x)\right] \right\vert  \notag \\
&&  \label{6} \\
&\leq &\frac{2^{\frac{1}{p}-1}}{\left( 2p+1\right) ^{\frac{1}{p}}\left(
b-a\right) ^{\frac{1}{p}}}\left[ \left( b-x\right) ^{2p+1}+\left( x-\frac{a+b%
}{2}\right) ^{2p+1}\right] ^{\frac{1}{p}}\left( \frac{\left\vert f^{\prime
\prime }(a)\right\vert ^{q}+\left\vert f^{\prime \prime }(b)\right\vert ^{q}%
}{2}\right) ^{\frac{1}{q}}  \notag
\end{eqnarray}%
for any $x\in \left[ \frac{a+b}{2},b\right] .$
\end{theorem}

\begin{proof}
From Lemma $\ref{lm}$, by the definition $k\left( t\right) $ and using by H%
\"{o}lder's inequality, it follows that%
\begin{eqnarray}
&&\left\vert \dfrac{1}{b-a}\dint_{a}^{b}f(u)du-\frac{1}{2}\left[
f(x)+f(a+b-x)\right] +\frac{1}{2}(x-\frac{a+3b}{4})\left[ f^{\prime
}(x)-f^{\prime }(a+b-x)\right] \right\vert  \notag \\
&&  \label{7} \\
&\leq &\frac{\left( b-a\right) ^{2}}{2}\dint_{0}^{1}\left\vert k\left(
t\right) \right\vert \left\vert f^{\prime \prime }(ta+(1-t)b)\right\vert dt 
\notag \\
&&  \notag \\
&\leq &\frac{\left( b-a\right) ^{2}}{2}\left( \dint_{0}^{1}\left\vert
k\left( t\right) \right\vert ^{p}dt\right) ^{\frac{1}{p}}\left(
\dint_{0}^{1}\left\vert f^{\prime \prime }(ta+(1-t)b)\right\vert
^{q}dt\right) ^{\frac{1}{q}}.  \notag
\end{eqnarray}%
Since $\left\vert f^{\prime \prime }\right\vert ^{q}$ is convex on $\left[
a,b\right] ,$ we know that for $t\in \lbrack 0,1]$%
\begin{equation*}
\left\vert f^{\prime \prime }(ta+(1-t)b)\right\vert ^{q}\leq t\left\vert
f^{\prime \prime }(a)\right\vert ^{q}+(1-t)\left\vert f^{\prime \prime
}(b)\right\vert ^{q},
\end{equation*}%
hence, a simple computation shows that%
\begin{equation}
\dint_{0}^{1}\left\vert f^{\prime \prime }(ta+(1-t)b)\right\vert ^{q}dt\leq 
\frac{\left\vert f^{\prime \prime }(a)\right\vert ^{q}+\left\vert f^{\prime
\prime }(b)\right\vert ^{q}}{2}  \label{8}
\end{equation}%
also,%
\begin{eqnarray}
\dint_{0}^{1}\left\vert k\left( t\right) \right\vert ^{p}dt &=&\dint_{0}^{%
\frac{b-x}{b-a}}t^{2p}dt+\dint_{\frac{b-x}{b-a}}^{\frac{x-a}{b-a}}\left\vert
t-\frac{1}{2}\right\vert ^{2p}dt+\dint_{\frac{x-a}{b-a}}^{1}\left(
1-t\right) ^{2p}dt  \notag \\
&&  \label{9} \\
&=&\frac{2}{\left( 2p+1\right) \left( b-a\right) ^{2p+1}}\left[ \left(
b-x\right) ^{2p+1}+\left( x-\frac{a+b}{2}\right) ^{2p+1}\right] .  \notag
\end{eqnarray}%
Using (\ref{8}) and (\ref{9}) in (\ref{7}), we obtain (\ref{6}).
\end{proof}

\begin{corollary}[Perturbed Trapezoid inequality]
\label{c2} Under the assumptions Theorem $\ref{thm2}$ with $x=b,$ we have%
\begin{equation*}
\left\vert \int\limits_{a}^{b}f(u)du-\frac{b-a}{2}\left[ f(b)+f(a)\right] +%
\frac{\left( b-a\right) ^{2}}{8}\left[ f^{\prime }(b)-f^{\prime }(a)\right]
\right\vert \leq \frac{(b-a)^{3}}{8\left( 2p+1\right) ^{\frac{1}{p}}}\left( 
\frac{\left\vert f^{\prime \prime }(a)\right\vert ^{q}+\left\vert f^{\prime
\prime }(b)\right\vert ^{q}}{2}\right) ^{\frac{1}{q}}.
\end{equation*}
\end{corollary}

\begin{corollary}[Trapezoid inequality]
Under the assumptions Theorem $\ref{thm2}$ with $x=b$ and $f^{\prime
}(a)=f^{\prime }(b)$ in Theorem $\ref{thm2},$ we have 
\begin{equation}
\left\vert \dfrac{1}{b-a}\dint_{a}^{b}f(u)du-\dfrac{f\left( a\right)
+f\left( b\right) }{2}\right\vert \leq \frac{(b-a)^{2}}{8\left( 2p+1\right)
^{\frac{1}{p}}}\left( \frac{\left\vert f^{\prime \prime }(a)\right\vert
^{q}+\left\vert f^{\prime \prime }(b)\right\vert ^{q}}{2}\right) ^{\frac{1}{q%
}}.  \label{H12}
\end{equation}
\end{corollary}

\begin{corollary}[Midpoint inequality]
Under the assumptions Theorem $\ref{thm2}$ with $x=\frac{a+b}{2}$ in Theorem 
$\ref{thm2},$ we have 
\begin{equation}
\left\vert \dfrac{1}{b-a}\dint_{a}^{b}f(u)du-f\left( \frac{a+b}{2}\right)
\right\vert \leq \frac{(b-a)^{2}}{8\left( 2p+1\right) ^{\frac{1}{p}}}\left( 
\frac{\left\vert f^{\prime \prime }(a)\right\vert ^{q}+\left\vert f^{\prime
\prime }(b)\right\vert ^{q}}{2}\right) ^{\frac{1}{q}}.  \label{H13}
\end{equation}
\end{corollary}

\begin{theorem}
\label{thm3} Let $f:I\subset \mathbb{R}\rightarrow \mathbb{R}$ be a twice
differentiable function on $I^{\circ }$ such that $f^{\prime \prime }\in
L_{1}[a,b]$ where $a,b\in I,$ $a<b$. If $\left\vert f^{\prime \prime
}\right\vert ^{q}$ is convex on $[a,b]$ and\ $q\geq 1,$ then%
\begin{eqnarray}
&&\left\vert \dfrac{1}{b-a}\dint_{a}^{b}f(u)du-\frac{1}{2}\left[
f(x)+f(a+b-x)\right] +\frac{1}{2}(x-\frac{a+3b}{4})\left[ f^{\prime
}(x)-f^{\prime }(a+b-x)\right] \right\vert  \notag \\
&&  \label{E0} \\
&\leq &\frac{1}{3\left( b-a\right) }\left[ \left( b-x\right) ^{3}+\left( x-%
\frac{a+b}{2}\right) ^{3}\right] \left( \frac{\left\vert f^{\prime \prime
}(a)\right\vert ^{q}+\left\vert f^{\prime \prime }(b)\right\vert ^{q}}{2}%
\right) ^{\frac{1}{q}}  \notag
\end{eqnarray}%
for any $x\in \left[ \frac{a+b}{2},b\right] .$
\end{theorem}

\begin{proof}
From Lemma $\ref{lm}$, by the definition $k\left( t\right) $ and using by
power mean inequality, it follows that%
\begin{eqnarray}
&&\left\vert \dfrac{1}{b-a}\dint_{a}^{b}f(u)du-\frac{1}{2}\left[
f(x)+f(a+b-x)\right] +\frac{1}{2}(x-\frac{a+3b}{4})\left[ f^{\prime
}(x)-f^{\prime }(a+b-x)\right] \right\vert  \notag \\
&&  \label{E1} \\
&\leq &\frac{\left( b-a\right) ^{2}}{2}\dint_{0}^{1}\left\vert k\left(
t\right) \right\vert \left\vert f^{\prime \prime }(ta+(1-t)b)\right\vert dt 
\notag \\
&&  \notag \\
&\leq &\frac{\left( b-a\right) ^{2}}{2}\left( \dint_{0}^{1}\left\vert
k\left( t\right) \right\vert dt\right) ^{1-\frac{1}{q}}\left(
\dint_{0}^{1}\left\vert k\left( t\right) \right\vert \left\vert f^{\prime
\prime }(ta+(1-t)b)\right\vert ^{q}dt\right) ^{\frac{1}{q}}.  \notag
\end{eqnarray}%
Since $\left\vert f^{\prime \prime }\right\vert ^{q}$ is convex on $\left[
a,b\right] ,$ we know that for $t\in \lbrack 0,1]$%
\begin{equation*}
\left\vert f^{\prime \prime }(ta+(1-t)b)\right\vert ^{q}\leq t\left\vert
f^{\prime \prime }(a)\right\vert ^{q}+(1-t)\left\vert f^{\prime \prime
}(b)\right\vert ^{q},
\end{equation*}%
hence, by simple computation%
\begin{eqnarray}
\dint_{0}^{1}\left\vert k\left( t\right) \right\vert dt &=&\dint_{0}^{\frac{%
b-x}{b-a}}t^{2}dt+\dint_{\frac{b-x}{b-a}}^{\frac{x-a}{b-a}}\left\vert t-%
\frac{1}{2}\right\vert ^{2}dt+\dint_{\frac{x-a}{b-a}}^{1}\left( 1-t\right)
^{2}dt  \notag \\
&&  \label{E2} \\
&=&\frac{2}{3\left( b-a\right) ^{3}}\left[ \left( b-x\right) ^{3}+\left( x-%
\frac{a+b}{2}\right) ^{3}\right] ,  \notag
\end{eqnarray}%
and 
\begin{eqnarray}
&&\dint_{0}^{1}\left\vert k\left( t\right) \right\vert \left\vert f^{\prime
\prime }(ta+(1-t)b)\right\vert ^{q}dt  \notag \\
&&  \notag \\
&\leq &\dint_{0}^{1}\left\vert k\left( t\right) \right\vert \left(
t\left\vert f^{\prime \prime }(a)\right\vert ^{q}+(1-t)\left\vert f^{\prime
\prime }(b)\right\vert ^{q}\right) dt  \notag \\
&&  \notag \\
&=&\dint_{0}^{\frac{b-x}{b-a}}\left( t^{3}\left\vert f^{\prime \prime
}(a)\right\vert ^{q}+(t^{2}-t^{3})\left\vert f^{\prime \prime
}(b)\right\vert ^{q}\right) dt  \notag \\
&&  \notag \\
&&+\dint_{\frac{b-x}{b-a}}^{\frac{x-a}{b-a}}\left[ \left( t-\frac{1}{2}%
\right) ^{2}t\left\vert f^{\prime \prime }(a)\right\vert ^{q}+\left( t-\frac{%
1}{2}\right) ^{2}(1-t)\left\vert f^{\prime \prime }(b)\right\vert ^{q}\right]
dt  \notag \\
&&  \label{E3} \\
&&+\dint_{\frac{x-a}{b-a}}^{1}\left[ \left( t-1\right) ^{2}t\left\vert
f^{\prime \prime }(a)\right\vert ^{q}+(1-t)^{3}\left\vert f^{\prime \prime
}(b)\right\vert ^{q}\right] dt  \notag \\
&&  \notag \\
&=&\frac{\left( b-x\right) ^{4}}{4(b-a)^{4}}\left\vert f^{\prime \prime
}(a)\right\vert ^{q}+\left( \frac{\left( b-x\right) ^{3}}{3(b-a)^{3}}-\frac{%
\left( b-x\right) ^{4}}{4(b-a)^{4}}\right) \left\vert f^{\prime \prime
}(b)\right\vert ^{q}  \notag \\
&&  \notag \\
&&+\frac{1}{3(b-a)^{3}}\left( x-\frac{a+b}{2}\right) ^{3}\left\vert
f^{\prime \prime }(a)\right\vert ^{q}+\frac{1}{3(b-a)^{3}}\left( x-\frac{a+b%
}{2}\right) ^{3}\left\vert f^{\prime \prime }(b)\right\vert ^{q}  \notag \\
&&  \notag \\
&&+\left( \frac{\left( b-x\right) ^{3}}{3(b-a)^{3}}-\frac{\left( b-x\right)
^{4}}{4(b-a)^{4}}\right) \left\vert f^{\prime \prime }(a)\right\vert ^{q}+%
\frac{\left( b-x\right) ^{4}}{4(b-a)^{4}}\left\vert f^{\prime \prime
}(b)\right\vert ^{q}  \notag \\
&&  \notag \\
&=&\frac{1}{3\left( b-a\right) ^{3}}\left[ \left( b-x\right) ^{3}+\left( x-%
\frac{a+b}{2}\right) ^{3}\right] \left( \left\vert f^{\prime \prime
}(a)\right\vert ^{q}+\left\vert f^{\prime \prime }(b)\right\vert ^{q}\right)
.  \notag
\end{eqnarray}%
Using (\ref{E2}) and (\ref{E3}) in (\ref{E1}), we obtain (\ref{E0}).
\end{proof}

\begin{corollary}
\label{c3} Under the assumptions Theorem $\ref{thm3}$ with $x=b,$ we have%
\begin{equation*}
\left\vert \int\limits_{a}^{b}f(u)du-\frac{b-a}{2}\left[ f(b)+f(a)\right] +%
\frac{\left( b-a\right) ^{2}}{8}\left[ f^{\prime }(b)-f^{\prime }(a)\right]
\right\vert \leq \frac{(b-a)^{3}}{24}\left( \frac{\left\vert f^{\prime
\prime }(a)\right\vert ^{q}+\left\vert f^{\prime \prime }(b)\right\vert ^{q}%
}{2}\right) ^{\frac{1}{q}}.
\end{equation*}
\end{corollary}

\begin{corollary}
Under the assumptions Theorem $\ref{thm3}$ with $x=b$ and $f^{\prime
}(a)=f^{\prime }(b)$ in Theorem $\ref{thm3},$ we have 
\begin{equation}
\left\vert \dfrac{1}{b-a}\dint_{a}^{b}f(u)du-\dfrac{f\left( a\right)
+f\left( b\right) }{2}\right\vert \leq \frac{(b-a)^{2}}{24}\left( \frac{%
\left\vert f^{\prime \prime }(a)\right\vert ^{q}+\left\vert f^{\prime \prime
}(b)\right\vert ^{q}}{2}\right) ^{\frac{1}{q}}.  \label{E5}
\end{equation}
\end{corollary}

\begin{corollary}
Under the assumptions Theorem $\ref{thm3}$ with $x=\frac{a+b}{2}$ in Theorem 
$\ref{thm3},$ we have 
\begin{equation}
\left\vert \dfrac{1}{b-a}\dint_{a}^{b}f(u)du-f\left( \frac{a+b}{2}\right)
\right\vert \leq \frac{(b-a)^{2}}{24}\left( \frac{\left\vert f^{\prime
\prime }(a)\right\vert ^{q}+\left\vert f^{\prime \prime }(b)\right\vert ^{q}%
}{2}\right) ^{\frac{1}{q}}.  \label{E6}
\end{equation}
\end{corollary}

\section{Applications for special means}

Recall the following means:

(a) The arithmetic mean 
\begin{equation*}
A=A(a,b):=\dfrac{a+b}{2},\ a,b\geq 0;
\end{equation*}

(b) The geometric mean 
\begin{equation*}
G=G(a,b):=\sqrt{ab},\text{ }a,b\geq 0;
\end{equation*}

(c) The harmonic mean 
\begin{equation*}
H=H\left( a,b\right) :=\dfrac{2ab}{a+b},\ a,b>0;
\end{equation*}

(d) The logarithmic mean 
\begin{equation*}
L=L\left( a,b\right) :=\left\{ 
\begin{array}{ccc}
a & if & a=b \\ 
&  &  \\ 
\frac{b-a}{\ln b-\ln a} & if & a\neq b%
\end{array}%
\right. \text{, \ \ \ }a,b>0;
\end{equation*}

(e) The identric mean%
\begin{equation*}
I=I(a,b):=\left\{ 
\begin{array}{ccc}
a & if & a=b \\ 
&  &  \\ 
\frac{1}{e}\left( \frac{b^{b}}{a^{a}}\right) ^{\frac{1}{b-a}} & if & a\neq b%
\end{array}%
\right. \text{, \ \ \ }a,b>0;
\end{equation*}

(f) The $p-$logarithmic mean:

\begin{equation*}
L_{p}=L_{p}(a,b):=\left\{ 
\begin{array}{ccc}
\left[ \frac{b^{p+1}-a^{p+1}}{\left( p+1\right) \left( b-a\right) }\right] ^{%
\frac{1}{p}} & \text{if} & a\neq b \\ 
&  &  \\ 
a & \text{if} & a=b%
\end{array}%
\right. \text{, \ \ \ }p\in \mathbb{R\diagdown }\left\{ -1,0\right\} ;\;a,b>0%
\text{.}
\end{equation*}%
It is also known \ that $L_{p}$ is monotonically nondecreasing in $p\in 
\mathbb{R}$ with $L_{-1}:=L$ and $L_{0}:=I.$ The following simple
relationships are known in the literature%
\begin{equation*}
H\leq G\leq L\leq I\leq A.
\end{equation*}%
Now, using the results of Section 2, some new inequalities is derived for
the above means.

\begin{proposition}
Let $p>1$ and $0\leq a<b.$ Then we have the inequality:%
\begin{equation*}
\left\vert L_{p}^{p}(a,b)-A\left( a^{p},b^{p}\right) \right\vert \leq
p\left( p-1\right) \frac{\left( b-a\right) ^{2}}{24}A\left(
a^{p-2},b^{p-2}\right) .
\end{equation*}
\end{proposition}

\begin{proof}
The assertion follows from (\ref{H9}) applied for $f(x)=x^{p},$ $x\in \left[
a,b\right] .$We omitted the details.
\end{proof}

\begin{proposition}
Let $p>1$ and $0\leq a<b.$ Then we have the inequality:%
\begin{equation*}
\left\vert L^{-1}(a,b)-A^{-1}(a,b)\right\vert \leq \frac{\left( b-a\right)
^{2}}{12}A\left( a^{-3},b^{-3}\right) .
\end{equation*}
\end{proposition}

\begin{proof}
The assertion follows from (\ref{H10}) applied for $f(x)=\frac{1}{x},$ $x\in %
\left[ a,b\right] .$ We omitted the details.
\end{proof}

\begin{proposition}
Let $p>1$ and $0\leq a<b.$ Then we have the inequality:%
\begin{equation*}
\left\vert \ln I(a,b)-\ln G(a,b)\right\vert \leq \frac{\left( b-a\right) ^{2}%
}{8\left( 2p+1\right) ^{\frac{1}{p}}}\left[ A\left( a^{-2q},b^{-2q}\right) %
\right] ^{1/q}.
\end{equation*}
\end{proposition}

\begin{proof}
The assertion follows from (\ref{H12}) applied for $f(x)=-\ln x,$ $x\in %
\left[ a,b\right] .$
\end{proof}

\begin{proposition}
Let $p>1$ and $0\leq a<b.$ Then we have the inequality:%
\begin{equation*}
\left\vert L_{p}^{p}(a,b)-A^{p}(a,b)\right\vert \leq p\left( p-1\right) 
\frac{\left( b-a\right) ^{2}}{8\left( 2p+1\right) ^{1/p}}\left[ A\left(
a^{q(p-2)},b^{q(p-2)}\right) \right] ^{1/q}.
\end{equation*}
\end{proposition}

\begin{proof}
The assertion follows from (\ref{H13}) applied for $f(x)=x^{p},$ $x\in \left[
a,b\right] .$
\end{proof}

\begin{proposition}
Let $p>1$ and $0\leq a<b.$ Then we have the inequality:%
\begin{equation*}
\left\vert L^{-1}(a,b)-H^{-1}(a,b)\right\vert \leq \frac{\left( b-a\right)
^{2}}{12}\left[ A\left( a^{-3q},b^{-3q}\right) \right] ^{1/q}.
\end{equation*}
\end{proposition}

\begin{proof}
The assertion follows from (\ref{E5}) applied for $f(x)=\frac{1}{x},$ $x\in %
\left[ a,b\right] .$
\end{proof}

\begin{proposition}
Let $p>1$ and $0\leq a<b.$ Then we have the inequality:%
\begin{equation*}
\left\vert \ln I(a,b)-\ln A(a,b)\right\vert \leq \frac{\left( b-a\right) ^{2}%
}{24}\left[ A\left( a^{-2q},b^{-2q}\right) \right] ^{1/q}.
\end{equation*}
\end{proposition}

\begin{proof}
The assertion follows from (\ref{E6}) applied for $f(x)=-\ln x,$ $x\in \left[
a,b\right] .$
\end{proof}

\section{Applications for composite quadrature formula}

Let $d$ be a division $a=x_{0}<x_{1}<...<x_{n-1}<x_{n}=b$ of the interval $%
\left[ a,b\right] $ and $\xi =(\xi _{0},...,\xi _{n-1})$ a sequence of
intermediate points, $\xi _{i}\in \lbrack x_{i},x_{i+1}],\ i=\overline{0,n-1}%
.$ Then the following result holds:

\begin{theorem}
Let $f:I\subset \mathbb{R}\rightarrow \mathbb{R}$ be a twice differentiable
function on $I^{\circ }$ such that $f^{\prime \prime }\in L_{1}[a,b]$ where $%
a,b\in I,$ $a<b$. If $\left\vert f^{\prime \prime }\right\vert $ is convex
on $[a,b]$ then we have 
\begin{equation*}
\dint_{a}^{b}f(u)du=A(f,f^{\prime },d,\xi )+R(f,f^{\prime },d,\xi )
\end{equation*}%
where 
\begin{eqnarray*}
A(f,f^{\prime },d,\xi ) &:&=\dsum\limits_{i=0}^{n-1}\frac{h_{i}}{2}\left[
f(\xi _{i})+f(x_{i}+x_{i+1}-\xi _{i})\right] \\
&& \\
&&-\dsum\limits_{i=0}^{n-1}\frac{h_{i}}{2}\left( \xi _{i}-\frac{%
x_{i}+3x_{i+1}}{4}\right) \left[ f^{\prime }(\xi _{i})-f^{\prime
}(x_{i}+x_{i+1}-\xi _{i})\right] .
\end{eqnarray*}%
The remainder $R(f,f^{\prime },d,\xi )$ satisfies the estimation:%
\begin{equation}
\left\vert R(f,f^{\prime },d,\xi )\right\vert \leq \dsum\limits_{i=0}^{n-1} 
\left[ \left( x_{i+1}-\xi _{i}\right) ^{3}+\left( \xi _{i}-\frac{%
x_{i}+x_{i+1}}{2}\right) ^{3}\right] \left( \frac{\left\vert f^{\prime
\prime }(x_{i})\right\vert +\left\vert f^{\prime \prime
}(x_{i+1})\right\vert }{6}\right)  \label{a1}
\end{equation}%
for any choice $\xi $ of the intermediate points.
\end{theorem}

\begin{proof}
Apply Theorem \ref{thm1} on the interval $[x_{i},x_{i+1}],\ \ i=\overline{%
0,n-1}$ to get%
\begin{multline*}
\left\vert \frac{h_{i}}{2}\left[ f(\xi _{i})+f(x_{i}+x_{i+1}-\xi _{i})\right]
-\frac{h_{i}}{2}\left( \xi _{i}-\frac{x_{i}+3x_{i+1}}{4}\right) \right. \\
\\
\left. \times \left[ f^{\prime }(\xi _{i})-f^{\prime }(x_{i}+x_{i+1}-\xi
_{i})\right] -\dint_{a}^{b}f(u)du\right\vert \\
\\
\leq \left[ \left( x_{i+1}-\xi _{i}\right) ^{3}+\left( \xi _{i}-\frac{%
x_{i}+x_{i+1}}{2}\right) ^{3}\right] \left( \frac{\left\vert f^{\prime
\prime }(x_{i})\right\vert +\left\vert f^{\prime \prime
}(x_{i+1})\right\vert }{6}\right) .
\end{multline*}
\end{proof}

Summing the above inequalities over $i$ from $0$ to $n-1$ and using the
generalized triangle inequality, we get the desired estimation (\ref{a1}).

\begin{corollary}
The following perturbed trapezoid rule holds:%
\begin{equation*}
\dint_{a}^{b}f(u)du=T(f,f^{\prime },d)+R_{T}(f,f^{\prime },d)
\end{equation*}%
where%
\begin{equation*}
T(f,f^{\prime },d):=\dsum\limits_{i=0}^{n-1}\frac{h_{i}}{2}\left[
f(x_{i})+f(x_{i+1})\right] -\dsum\limits_{i=0}^{n-1}\frac{\left(
h_{i}\right) ^{2}}{8}\left[ f^{\prime }(x_{i+1})-f^{\prime }(x_{i})\right]
\end{equation*}%
and the remainder term $R_{T}(f,f^{\prime },d)$ satisfies the estimation,%
\begin{equation*}
R_{T}(f,f^{\prime },d)\leq \dsum\limits_{i=0}^{n-1}\dfrac{(h_{i})^{3}}{48}%
(\left\vert f^{\prime \prime }(x_{i})\right\vert +\left\vert f^{\prime
\prime }(x_{i+1})\right\vert ).
\end{equation*}
\end{corollary}

\begin{corollary}
The following midpoint rule holds:%
\begin{equation*}
\dint_{a}^{b}f(u)du=M(f,d)+R_{M}(f,d)
\end{equation*}%
where%
\begin{equation*}
M(f,d):=\dsum\limits_{i=0}^{n-1}h_{i}\left[ f(\frac{x_{i}+x_{i+1}}{2})\right]
\end{equation*}%
and the remainder term $R_{M}(f,d)$ satisfies the estimation,%
\begin{equation*}
R_{M}(f,d)\leq \dsum\limits_{i=0}^{n-1}\dfrac{(h_{i})^{3}}{48}(\left\vert
f^{\prime \prime }(x_{i})\right\vert +\left\vert f^{\prime \prime
}(x_{i+1})\right\vert ).
\end{equation*}
\end{corollary}

\end{document}